\documentclass[12pt, a4paper]{article}
\usepackage[utf8]{inputenc}
\usepackage[T1]{fontenc}
\usepackage[french, english]{babel}
\usepackage{amsmath, amssymb, amsthm}
\usepackage{geometry}
\usepackage{authblk}
\usepackage{hyperref}

\newtheorem{lemma}{Lemma}
\newtheorem{theorem}{Theorem}
\newtheorem{remark}{Remark}

\title{ON CERTAIN SUMS INVOLVING THE LARGEST PRIME FACTOR OVER INTEGER SEQUENCES}
\author{Mihoub BOUDERBALA}

\date{}

\begin{document}

\maketitle

\begin{abstract}
\noindent\textbf{Abstract.} Given an integer $n \ge 2$, its prime factorization is expressed as $n= \prod_{i=1}^s p_i^{a_i}$. We define the function $f(n)$ as the smallest positive integer such that $f(n)!$ is divisible by $n$. The main objective of this paper is to derive an asymptotic formula for both sums $\sum_{n \le x} f(n)$ and $\sum_{n \le x, n \in S_k} f(n)$, where $S_k$ denotes the set of all $k$-free integers.

\vspace{0.5cm}
\noindent\textbf{Résumé.} Étant donné un entier $n \ge 2$, sa décomposition en facteurs premiers s'écrit $n= \prod_{i=1}^s p_i^{a_i}$. On définit la fonction $f(n)$ comme le plus petit entier positif tel que $f(n)!$ soit divisible par $n$. Ce papier vise surtout à trouver une formule asymptotique pour les deux sommes $\sum_{n \le x} f(n)$ et $\sum_{n \le x, n \in S_k} f(n)$, où $S_k$ désigne l'ensemble de tous les entiers $k$-libres.
\end{abstract}

\noindent\textbf{Keywords.} Summation formula, largest prime factor, $k$-free numbers.

\vspace{0.2cm}
\noindent\textbf{2010 Mathematics Subject Classification:} 11A05; 11N37.

\section{Introduction}
Throughout this article, all asymptotic estimates are understood as $x \to \infty$, unless otherwise specified. Given an integer $n \ge 2$, its decomposition into prime factors is written as $n= p_1^{a_1} p_2^{a_2} \dots p_s^{a_s}$, where $p_1 < p_2 < \dots < p_s$ are primes and $a_1, a_2, \dots, a_s$ are positive integers. We define the function $f(n)$ as the smallest positive integer such that $f(n)!$ is divisible by $n$. Additionally, we define the function $P(n)$, which denotes the largest prime factor in the factorization of the positive integer $n$. The study of sums involving $f(n)$ is not only of intrinsic arithmetic interest, but also reveals deeper connections with the anatomy of integers, the distribution of smooth numbers, and mean value theorems for multiplicative functions. Since $f(n)$ is intrinsically linked to $P(n)$ as shown in Lemma 1 below. Our asymptotic estimates naturally extend classical results on sums weighted by $P(n)$, such as those of Alladi and Erdős. Furthermore, the appearance of the zeta values $\zeta(2)$ and $\zeta(2k)$ in our main theorems reflects the underlying structure of Euler's product, suggesting potential links with Dirichlet series and $L$-functions associated with factorial divisibility.

The function $P(n)$ appeared in the article by K. Alladi and P. Erdős, published in 1977 [1], in which a number of its properties were studied, in particular its asymptotic behaviour, where they were able to show that for any real number $x \ge 1$, we hold
\begin{equation}
\sum_{n \le x} P(n) = \frac{\pi^2}{12} \frac{x^2}{\log x} + O\left( \frac{x^2}{\log^2 x} \right). \tag{1}
\end{equation}
Given fixed integers $k \ge 2$, we say that $n$ is an $k$-free number if $\max(a_1, a_2, \dots, a_s) < k$.
We will denote by $S_k$ the set of $k$-free integers. Let us note that the estimates for the counting function $S_k(x) := \text{card}\{n \in \mathbb{N} \mid n \le x \text{ and } n \in S_k\}$, of this family of numbers are already known. Furthermore, for fixed integers $k \ge 2$, they are given by
\begin{equation}
S_k(x) = \frac{x}{\zeta(k)} + O\left(x^{\frac{1}{k}}\right). \tag{2}
\end{equation}
The proof of this formula (2), in the simplest case, i.e., for $k=2$, can already be found in the book by Niven, Zuckerman, and Montgomery [7, Theorem 8.25]. A complete treatment for arbitrary $k \ge 2$ can be found in the survey paper by Pappalardi [8], which confirms the uniformity of the argument across all $k \ge 2$. In this work, we focus on deriving asymptotic formulas for two specific sums: $\sum_{n \le x} f(n)$ and $\sum_{n \le x, n \in S_k} f(n)$, where $S_k$ denotes the set of all $k$-free integers. These sums are closely related to the arithmetic function $P(n)$, which plays a central role in our analysis. The main results of this study, obtained primarily through rigorous analytical methods, are presented as theorems in the following section. These theorems are built upon two preliminary lemmas, which provide the necessary technical tools for their proofs.

\section{Main results}

\begin{lemma}
Given an integer $n \ge 2$, let $P(n)$ stand for its largest prime factor, with $P(1)=1$. If $P(n)^2 > n$ we have $f(n)= P(n)$.
\end{lemma}

\begin{proof}
Let $n$ be an integer such that $n \ge 2$, and $n= p_1^{a_1} p_2^{a_2} \dots p_s^{a_s}$. We have $P(n)= p_s$. If $P(n)^2 > n$, it necessarily follows $a_s= 1$, and we have
$$ p_1^{a_1} p_2^{a_2} \dots p_{s-1}^{a_{s-1}} < n^{\frac{1}{2}} < P(n)= p_s. $$
Moreover, for all $i= 1, 2, \dots, s-1$, we have
$$ p_i^{a_i} < p_s \Rightarrow p_i^{a_i} \mid p_s!, $$
which implies
$$ p_1^{a_1} p_2^{a_2} \dots p_{s-1}^{a_{s-1}} \mid p_s!. $$
Therefore, for each $i= 1, 2, \dots, s-1$, we have $p_i^{a_i} \mid p_s!$ and since $a_s= 1$, we also have $p_s \mid p_s!$.
Consequently, it follows that
\begin{equation}
n \mid p_s!. \tag{3}
\end{equation}
But since, for any prime number $p$, $p$ can never be found in the decomposition of $(p-1)!$, it follows that the smallest prime number verifying (3) is $p_s$, hence $f(n)= p_s = P(n)$.
\end{proof}

\begin{lemma}
Let $x$ be a real number such that $x > 1$, and $\delta$ be any real function such that $\forall n \in \mathbb{N}$, $1 \le \delta(n) \le 1$, we have
$$ \sum_{n \le x^{\frac{1}{2}}} \frac{\delta(n)}{n^2 \log(x/n)} = \frac{1}{\log x} \sum_{n=1}^{\infty} \frac{\delta(n)}{n^2} + O\left( \frac{1}{\log^2 x} \right). $$
\end{lemma}

\begin{proof}
Since $n \le x^{\frac{1}{2}}$, it follows that $\frac{\log n}{\log x} \le \frac{1}{2}$. Therefore,
\begin{align*}
\frac{\delta(n)}{n^2 \log(x/n)} &= \frac{\delta(n)}{n^2 \log x \left(1 - \frac{\log n}{\log x}\right)} \\
&= \frac{\delta(n)}{n^2 \log x} \left(1 + O\left(\frac{\log n}{\log x}\right)\right) \\
&= \frac{\delta(n)}{n^2 \log x} + O\left(\frac{\delta(n) \log n}{n^2 \log^2 x}\right).
\end{align*}
This means,
$$ \sum_{n \le x^{\frac{1}{2}}} \frac{\delta(n)}{n^2 \log(x/n)} = \frac{1}{\log x} \sum_{n \le x^{\frac{1}{2}}} \frac{\delta(n)}{n^2} + O\left( \frac{1}{\log^2 x} \sum_{n=1}^{\infty} \frac{\delta(n) \log n}{n^2} \right). $$
Noting that, the associated series $\sum_{n=1}^{\infty} \frac{\delta(n) \log n}{n^2}$ is convergent and
\begin{align*}
\frac{1}{\log x} \sum_{n \le x^{\frac{1}{2}}} \frac{\delta(n)}{n^2} &= \frac{1}{\log x} \sum_{n=1}^{\infty} \frac{\delta(n)}{n^2} - \frac{1}{\log x} \sum_{n > x^{\frac{1}{2}}} \frac{\delta(n)}{n^2} \\
&= \frac{1}{\log x} \sum_{n=1}^{\infty} \frac{\delta(n)}{n^2} + O\left( \frac{1}{\log x} \int_{x^{\frac{1}{2}}}^{\infty} \frac{dt}{t^2} \right) \\
&= \frac{1}{\log x} \sum_{n=1}^{\infty} \frac{\delta(n)}{n^2} + O\left( \frac{1}{x^{\frac{1}{2}} \log x} \right).
\end{align*}
Leading
$$ \sum_{n \le x^{\frac{1}{2}}} \frac{\delta(n)}{n^2 \log(x/n)} = \frac{1}{\log x} \sum_{n=1}^{\infty} \frac{\delta(n)}{n^2} + O\left( \frac{1}{\log^2 x} \right). $$
\end{proof}

\begin{theorem}
Let $x$ be any real number such that $x \ge 1$. We have
$$ \sum_{n \le x} f(n) = \zeta(2) \frac{x^2}{\log x} + O\left( \frac{x^2}{\log^2 x} \right). $$
\end{theorem}

\begin{proof}
The idea of proving the theorem is to start by seeing that
\begin{equation}
\sum_{n \le x} f(n) = \sum_{n \le x, P(n)^2 \le n} f(n) + \sum_{n \le x, P(n)^2 > n} f(n). \tag{4}
\end{equation}
In order to find an estimate for the first subsum appearing in the principal sum (4), we begin by noting that $f(n) \ll P(n) \log n$. This follows, for instance, from the bound $f(n) = \max_{p^k \mid n} f(p) \le P(n) \log_2 n$, as shown in [5, section 1.4]. A sharper bound $f(n) \ll P(n) \log n \log P(n) + P(n)$ can be found in [6] and [4]. Consequently, we have
$$ \sum_{n \le x, P(n)^2 \le n} f(n) \ll \sum_{n \le x, P(n)^2 \le n} P(n) \log n \le \sum_{n \le x} n^{\frac{1}{2}} \log n, $$
and it suffices to estimate this latter sum. We proceed with a simpler and more direct approach, we observe that
$$ \sum_{n \le x} n^{\frac{1}{2}} \log n \le x^{\frac{1}{2}} \log x \sum_{n \le x} 1 \ll x^{\frac{3}{2}} \log x. $$
From this, it follows immediately that
\begin{equation}
\sum_{n \le x, P(n)^2 \le n} f(n) \ll x^{\frac{3}{2}} \log x. \tag{5}
\end{equation}
To get an explicit estimation at the second subsum of the sum (4), we first use the result obtained in the previous lemma, which it leads to
$$ \sum_{n \le x, P(n)^2 > n} f(n) = \sum_{n \le x, P(n)^2 > n} P(n) = \sum_{mp \le x, p^2 > mp} p. $$
We can see that if $P(n)^2 > n$, then $n^{\frac{1}{2}} < p$. This implies that $n^{\frac{1}{2}} m < pm = n$, so $m < n^{\frac{1}{2}} < x^{\frac{1}{2}}$. Therefore, we conclude that
\begin{align*}
\sum_{n \le x, P(n)^2 > n} f(n) &= \sum_{m \le x^{\frac{1}{2}}} \sum_{m < p \le \frac{x}{m}} p \\
&= \sum_{m \le x^{\frac{1}{2}}} \sum_{m < p \le x^{\frac{1}{2}}} p + \sum_{m \le x^{\frac{1}{2}}} \sum_{x^{\frac{1}{2}} < p \le \frac{x}{m}} p,
\end{align*}
and hence
\begin{equation}
\sum_{n \le x, P(n)^2 > n} f(n) = \sum_{m \le x^{\frac{1}{2}}} \sum_{x^{\frac{1}{2}} < p \le \frac{x}{m}} p + O\left( \sum_{m \le x^{\frac{1}{2}}} \sum_{m < p \le x^{\frac{1}{2}}} x^{\frac{1}{2}} \right). \tag{6}
\end{equation}
We now recall the function $\pi(x)$, which counts all prime numbers less than or equal to $x$, where
\begin{equation}
\pi(x) = \frac{x}{\log x} + O\left( \frac{x}{\log^2 x} \right). \quad (\text{see [3, p.119]}) \tag{7}
\end{equation}
Thus, based on the latter, we can see that
\begin{equation}
\sum_{m \le x^{\frac{1}{2}}} \sum_{m < p \le x^{\frac{1}{2}}} x^{\frac{1}{2}} \le x^{\frac{1}{2}} \sum_{m \le x^{\frac{1}{2}}} \pi(x^{\frac{1}{2}}) \ll \frac{x^{\frac{3}{2}}}{\log x}, \tag{8}
\end{equation}
and
$$ \sum_{x^{\frac{1}{2}} < p \le \frac{x}{m}} p = \sum_{x^{\frac{1}{2}} < n \le \frac{x}{m}} n (\pi(n) - \pi(n-1)), $$
where $g(n) = \pi(n) - \pi(n-1)$, stands for the prime number indicator function (with a value of 1 if $n$ is prime, otherwise 0) and $\sum_{n \le y} g(n) = \sum_{p \le y} p = \pi(y)$.
Using the formula (7), and Abel's summation formula [2, p.77], we obtain
$$ \sum_{x^{\frac{1}{2}} < p \le \frac{x}{m}} p = \pi\left(\frac{x}{m}\right) \frac{x}{m} - \pi(x^{\frac{1}{2}}) x^{\frac{1}{2}} - \int_{x^{\frac{1}{2}}}^{\frac{x}{m}} \pi(t) \, dt. $$
Furthermore,
\begin{equation}
\sum_{x^{\frac{1}{2}} < p \le \frac{x}{m}} p = \frac{1}{2} \frac{x^2}{m^2 \log \frac{x}{m}} - \frac{1}{2} \frac{x}{\log x^{\frac{1}{2}}} + O\left( \frac{x^2}{m^2 \log^2 \frac{x}{m}} \right) + O\left( \frac{x}{\log^2 x} \right). \tag{9}
\end{equation}
We now proceed by combining elementary estimates with a Taylor expansion. Since $n \le x^{\frac{1}{2}}$, we have $\log n < \log x$ and thus $\log\left(\frac{x}{n}\right) = \log x \left(1 - \frac{\log n}{\log x}\right)$.
Hence,
\begin{align*}
\frac{x^2}{n^2 \log \frac{x}{n}} &= \frac{x^2}{n^2 \log x \left(1 - \frac{\log n}{\log x}\right)} \\
&= \frac{x^2}{n^2 \log x} \left(1 + O\left(\frac{\log n}{\log x}\right)\right) \\
&= \frac{x^2}{n^2 \log x} + O\left( \frac{x^2 \log n}{n^2 \log^2 x} \right).
\end{align*}
We obtain,
$$ \sum_{n \le x^{\frac{1}{2}}} \frac{x^2}{n^2 \log(x/n)} = \frac{x^2}{\log x} \sum_{n \le x^{\frac{1}{2}}} \frac{1}{n^2} + O\left( \frac{x^2}{\log^2 x} \sum_{n=1}^{\infty} \frac{\log n}{n^2} \right). $$
Noting that the series $\sum_{n=1}^{\infty} \frac{\log n}{n^2}$ converges (its value is $-\zeta'(2)$). Moreover, using the estimate $\sum_{n \le x} \frac{1}{n^2} = \zeta(2) + O(x^{-1})$, we get
\begin{align*}
\frac{x^2}{\log x} \sum_{n \le x^{\frac{1}{2}}} \frac{1}{n^2} &= \frac{x^2}{\log x} \left( \zeta(2) + O(x^{-\frac{1}{2}}) \right) \\
&= \frac{x^2}{\log x} \zeta(2) + O\left( \frac{x^{\frac{3}{2}}}{\log x} \right).
\end{align*}
Leading
$$ \sum_{n \le x^{\frac{1}{2}}} \frac{\log x}{n^2 \log(x/n)} = \frac{x^2}{\log x} \zeta(2) + O\left( \frac{x^{\frac{3}{2}}}{\log x} \right) + O\left( \frac{x^2}{\log^2 x} \right). $$
and similarly
$$ \sum_{m \le x^{\frac{1}{2}}} \frac{x^2}{m^2 \log^2 \frac{x}{m}} \ll \frac{x^2}{\log^2 x}. $$
Using these last two results and formula (9), we obtain
\begin{align}
\sum_{m \le x^{\frac{1}{2}}} \sum_{x^{\frac{1}{2}} < p \le \frac{x}{m}} p &= \zeta(2) \frac{x^2}{\log x} + O\left( \frac{x^{\frac{3}{2}}}{\log x} \right) + O\left( \frac{x^2}{\log^2 x} \right) + O\left( \frac{x^{\frac{3}{2}}}{\log^2 x} \right) \nonumber \\
&= \zeta(2) \frac{x^2}{\log x} + O\left( \frac{x^2}{\log^2 x} \right). \tag{10}
\end{align}
The conjunction of formulas (10) and (6) gives us
\begin{equation}
\sum_{n \le x, P(n)^2 > n} f(n) = \zeta(2) \frac{x^2}{\log x} + O\left( \frac{x^2}{\log^2 x} \right). \tag{11}
\end{equation}
Finally, by substituting formulae (11), (8) and (5) into (4), we complete the proof of Theorem 3.
\end{proof}

\begin{theorem}
Let $x > 1$, be a real number. Then, given any positive integer $k$, such that $k \ge 2$, we have
$$ \sum_{n \le x, n \in S_k} f(n) = \frac{\zeta^2(2)}{2\zeta(2k)} \frac{x^2}{\log x} + O\left( \frac{x^2}{\log^2 x} \right). $$
\end{theorem}

\begin{proof}
Firstly, consider $\delta_k(n)$ as the characteristic function of the $k$-free numbers, meaning that,
$$ \delta_k(n) = \begin{cases} 1 & \text{if } n \text{ is } k\text{-free}, \\ 0 & \text{otherwise}. \end{cases} $$
In particular, its generating function for $s > 1$ is given by
\begin{equation}
\sum_{n=1}^{\infty} \frac{\delta_k(n)}{n^s} = \prod_{p} \left( 1 + \frac{1}{p^s} + \dots + \frac{1}{p^{(k-1)s}} \right) = \frac{\zeta(s)}{\zeta(ks)}. \tag{12}
\end{equation}
To begin with, we have
\begin{align}
M(x) &:= \sum_{n \le x, n \in S_k} f(n) = \sum_{n \le x} \delta_k(n) f(n) \nonumber \\
&= \sum_{n \le x, P(n)^2 \le n} \delta_k(n) f(n) + \sum_{n \le x, P(n)^2 > n} \delta_k(n) f(n) \nonumber \\
&= M_1(x) + M_2(x). \tag{13}
\end{align}
To establish an estimate for the sum $M_2(x)$, we rely on the result of Lemma 1, hence
\begin{align*}
M_2(x) &= \sum_{pm \le x, p^2 > mp} \delta_k(m) p \\
&= \sum_{m \le x^{\frac{1}{2}}} \delta_k(m) \sum_{m < p \le \frac{x}{m}} p.
\end{align*}
To proceed properly, we have
\begin{align}
M_2(x) &= \sum_{m \le x^{\frac{1}{2}}} \delta_k(m) \sum_{x^{\frac{1}{2}} < p \le \frac{x}{m}} p + \sum_{m \le x^{\frac{1}{2}}} \delta_k(m) \sum_{m < p \le x^{\frac{1}{2}}} p \nonumber \\
&= \sum_{m \le x^{\frac{1}{2}}} \delta_k(m) \sum_{x^{\frac{1}{2}} < p \le \frac{x}{m}} p + O\left( \frac{x^{\frac{3}{2}}}{\log x} \right). \tag{14}
\end{align}
And more
\begin{align*}
\sum_{m \le x^{\frac{1}{2}}} \delta_k(m) \sum_{x^{\frac{1}{2}} < p \le \frac{x}{m}} p &= \sum_{m \le x^{\frac{1}{2}}} \delta_k(m) \sum_{p \le \frac{x}{m}} p - \sum_{m \le x^{\frac{1}{2}}} \delta_k(m) \sum_{p \le x^{\frac{1}{2}}} p \\
&= \sum_{m \le x^{\frac{1}{2}}} \delta_k(m) \sum_{p \le \frac{x}{m}} p + O\left( \frac{x^{\frac{3}{2}}}{\log x} \right).
\end{align*}
By using the result noted as (1), but this time replacing $x$ with $x/m$, we obtain
\begin{align*}
\sum_{m \le x^{\frac{1}{2}}} \delta_k(m) \sum_{x^{\frac{1}{2}} < p \le \frac{x}{m}} p &= \sum_{m \le x^{\frac{1}{2}}} \delta_k(m) \left( \frac{\pi^2}{12} \frac{(x/m)^2}{\log(x/m)} + O\left( \frac{(x/m)^2}{\log^2(x/m)} \right) \right) + O\left( \frac{x^{\frac{3}{2}}}{\log x} \right) \\
&= \frac{\pi^2 x^2}{12} \sum_{m \le x^{\frac{1}{2}}} \frac{\delta_k(m)}{m^2 \log(x/m)} + O\left( \sum_{m \le x^{\frac{1}{2}}} \frac{x^2}{m^2 \log^2(x/m)} \right) + O\left( \frac{x^{\frac{3}{2}}}{\log x} \right) \\
&= \frac{\pi^2 x^2}{12} \sum_{m \le x^{\frac{1}{2}}} \frac{\delta_k(m)}{m^2 \log(x/m)} + O\left( \frac{x^2}{\log^2 x} \right).
\end{align*}
Using the result obtained in Lemma 2, with (12) it gives us
\begin{align*}
\sum_{m \le x^{\frac{1}{2}}} \delta_k(m) \sum_{x^{\frac{1}{2}} < p \le \frac{x}{m}} p &= \frac{\pi^2 x^2}{12} \left( \frac{1}{\log x} \sum_{n=1}^{\infty} \frac{\delta_k(n)}{n^2} + O\left( \frac{1}{\log^2 x} \right) \right) + O\left( \frac{x^2}{\log^2 x} \right) \\
&= \frac{\zeta^2(2)}{2\zeta(2k)} \frac{x^2}{\log x} + O\left( \frac{x^2}{\log^2 x} \right).
\end{align*}
We substitute in (14) and obtain
\begin{equation}
M_2(x) = \frac{\zeta^2(2)}{2\zeta(2k)} \frac{x^2}{\log x} + O\left( \frac{x^2}{\log^2 x} \right). \tag{15}
\end{equation}
It is obvious that
$$ M_1(x) \le \sum_{n \le x, P(n)^2 \le n} f(n) $$
so from (5) it follows that
\begin{equation}
M_1(x) = O\left( x^{\frac{3}{2}} \log x \right). \tag{16}
\end{equation}
The conjunction of formulas (16), (15) and (13) completes the proof of Theorem 4.
\end{proof}

\begin{remark}
The methods developed in this paper can be extended to study higher moments of the form
$$ \sum_{n \le x} f(n)^r \quad \text{and} \quad \sum_{n \le x, n \in S_k} f(n)^r, $$
for fixed $r \ge 2$. Since $f(n) \ll P(n) \log n$, and moments of $P(n)$ are well-understood (see, e.g., [1]), one expects asymptotic formulas of the shape
$$ \sum_{n \le x} f(n)^r \sim C_r \frac{x^{r+1}}{(\log x)^r} \quad \text{as } x \to \infty, $$
for some constants $C_r > 0$.
\end{remark}

\section*{Acknowledgement}
I would like to express my sincere gratitude to Professor Olivier Bordellès for his interest in this work and for his insightful and constructive comments, which significantly improved its presentation and rigor. I am also deeply grateful to the referees for their thoughtful suggestions and valuable critiques, which helped enhance the clarity and depth of this paper.

\section*{References}
\begin{enumerate}
    \item K. Alladi and P. Erdős, On an additive arithmetic function, Pacific J. Math. 71 no. 2 (1977), 275--294.
    \item T. Apostol, Introduction to Analytic Number Theory, New York, 1976.
    \item O. Bordellès, Arithmetic Tales, Springer, 2020.
    \item P. Erdős and A. Ivić, Estimates for sums involving the largest prime factor of an integer and certain related additive functions, Studia Sci, Math. Hungar, 15 (1980), 183--199.
    \item A. Ivić, The Riemann Zeta-Function: Theory and Applications, Dover, 2003.
    \item J. B. Rosser and L. Schoenfeld, Approximate formulas for some functions of prime numbers, Illinois J. Math, 6 (1962), 64--94.
    \item I. Niven, H. S. Zuckerman, and H. L. Montgomery, An introduction to the theory of numbers, fifth ed., John Wiley \& Sons, New York, 1991.
    \item F. Pappalardi, A survey on $k$-freeness, in Number theory, Ramanujan Math. Soc. Lect. Notes Ser. 1, Ramanujan Math. Soc., Mysore, 2005, pp. 71--88.
\end{enumerate}

\vspace{1cm}
\noindent\textbf{BOUDERBALA, Mihoub} \\
(1) Faculty of matter sciences and Computer Sciences Department of Mathematics FIMA Laboratory, Khemis Miliana University (UDBKM) Rue Thniet El Had, Khemis Miliana, 44225, Ain Defa Province, Algeria \\
(2) National Higher School of Mathematics, P.O. Box 75, Scientific and Technology Hub of Sidi Abdellah, 16093, Algiers, Algeria \\
E-mail address: mihoub75bouder@gmail.com

\end{document}